\documentclass[12pt]{amsart} 
\usepackage{amssymb,latexsym} 
\usepackage{graphicx} 
\usepackage{amsmath,amsthm}
\usepackage{amssymb}
\usepackage{enumitem}
\usepackage{amscd}
\usepackage{amsfonts}
\usepackage{amsbsy}
\usepackage{epsfig}
\usepackage[pagewise]{lineno}
\usepackage{soul}
\usepackage[normalem]{ulem}
\usepackage{soul}
\usepackage{enumitem}
\usepackage{tikz}
\newtheorem{theorem}{Theorem}
\newtheorem{lemma}[theorem]{Lemma}

\newtheorem{proposition}[theorem]{Proposition}
\newtheorem{corollary}[theorem]{Corollary}

\newtheorem{remark}[theorem]{Remark}

\newtheorem*{theorem*}{Theorem}

\newtheorem*{claim*}{Claim}

\newfont\bbf{msbm10 at 12pt}

\def\eps{\varepsilon}

\def\N{{\mathbb N}}
\def\Z{{\mathbb Z}}

\def\P{{\mathcal P}}

\def\R{{\mathcal R}}

\def\min{{\hbox{{\rm min}}}}
\def\max{{\hbox{{\rm max}}}}

\def\Per{{\hbox{{\rm Per}}}}
\def\id{{\hbox{{\rm id}}}}

\def\le{\leqslant}
\def\ge{\geqslant}

\def\1{ {\hbox{{\it 1}} \!\! I} }

\def\eps{\varepsilon}
\newcommand{\Aut}{\mathrm{Aut}}
\def\Cl{C_\lambda}

\begin{document}
\title[Nowhere monotonicity and invertibility a.e.]
{Generic continuous Lebesgue measure-preserving interval maps are nowhere monotone but invertible a.e.}
\author[Bobok]{Jozef Bobok}
\author[\v Cin\v c]{Jernej \v Cin\v c}
\author[Oprocha]{Piotr Oprocha}
\author[Troubetzkoy]{Serge Troubetzkoy}

\date{\today}

\address[J.\ Bobok]{Department of Mathematics of FCE, Czech Technical University in Prague,
Th\'akurova 7, 166 29 Prague 6, Czech Republic}
\email{jozef.bobok@cvut.cz}

\address[J.\ \v{C}in\v{c}]{University of Maribor, Koro\v ska 160, 2000 Maribor, Slovenia
 -- $\&$ -- 
	Abdus Salam International Centre for Theoretical Physics (ICTP), Trieste, Italy}
\email{jernej.cinc@um.si}

\address[P.\ Oprocha]{
Centre of Excellence IT4Innovations - Institute for Research and Applications of Fuzzy Modeling, University of Ostrava, 30. dubna 22, 701 03 Ostrava 1, Czech Republic}
\email{piotr.oprocha@osu.cz}

\address[S.\ Troubetzkoy]{Aix Marseille Univ, CNRS, I2M, Marseille, France}
\email{serge.troubetzkoy@univ-amu.fr}

\begin{abstract}
We consider continuous maps of the interval which preserve the Lebesgue measure.  
Except for the identity map or $1 - \id$ all such maps have topological entropy at least $\log2/2$ and generically they have infinite topological entropy.  
In this article we show that the
generic map with respect to the uniform topology has zero measure-theoretic
entropy with respect to the Lebesgue measure. This implies that there are dramatic differences in
the topological versus measure-theoretic behavior both for injectivity  as well as for the structure of the
level sets of generic maps. As a consequence we get a surprising corollary for a family of planar attractors homeomorphic to the pseudo-arcs.
\end{abstract}

\maketitle
\section{Introduction}
Genericity of zero measure-theoretic entropy has a long history starting with a result of Rokhlin.
Let $(X,\mu)$ denote a Lebesgue probability space and $\Aut(X,\mu)$ the group of measure preserving
automorphisms of $X$.  Consider the weak topology:
a sequence $\{f_i\} \subset \Aut(X,\mu)$ converges to $f$ if for every measurable subset $A \subset X$ we have
$\mu(f_i(A) \triangle f(A)) \to 0$ as $i \to \infty$
where $\triangle$ denotes the symmetric difference of sets.  This makes $\Aut(X,\mu)$  into a Polish space, i.e., a separable completely metrizable topological space.
In 1959  Rokhlin showed that  the subset of all transformations $T \in \Aut(X, \mu)$ that have zero entropy with
respect to $\mu$ contains a dense $G_\delta$ subset \cite{Ro}.
Later  Rudolf extended this result to countable amenable actions \cite{FoWe}  and  L.\ Bowen further extended this to the non-amenable case \cite{Bo}.

In the late 60s Katok and Stepin developed a quantitative approximation technique and were able to
apply it not only for automorphisms but also for homeomorphisms of a compact manifold $M$ of dimension at least 2 with respect to the uniform topology, i.e.,
induced by the metric
\begin{equation*}
\rho(f,g) := \sup_{x \in M} d(f(x),g(x))
\end{equation*}
where $d$ denotes the standard metric on $M$.
In particular, they showed that the generic Lebesgue  measure-preserving homeomorphisms of $M$ have zero  measure-theoretic entropy with respect to the Lebesgue measure
\cite{KS1,KS2}. These results have modern proofs established by Guih\'eneuf \cite{PAG}.

Zero entropy systems are invertible almost surely, thus such genericity results as described above might not be surprising for invertible systems. It is less clear that similar results 
could hold in the non-invertible setting.  There are only a few known results for non-invertible maps in this direction. In particular, Parry proved the following results  in \cite{Pa}.  In the setting of  the one-sided full 2-shift, Parry considered the weak-topology on the set of invariant  ergodic probability 
measures, and showed that the set of zero entropy (ergodic, fully supported, and non-atomic)  measures is a dense $G_\delta$ set. 
This result immediately translates to the doubling map on the circle.
Then he considered the space $S$ of strictly increasing degree 2 maps of the circle which preserve the Lebesgue measure with the uniform topology.  In this setting he showed  that the set of 
ergodic maps whose  metric entropy with respect to Lebesgue measure is zero is dense $G_\delta$.  

It is clear that Parry's result also holds for full branch unimodal interval maps preserving the Lebesgue measure. The
only other result in this direction is the  generalization of Parry's result to full branch $k$-modal interval maps preserving Lebesgue measure by Friedland and Weiss \cite{FrWe}.

We consider the generalization of this question when we  omit  the (quite strong) piecewise monotone restriction. 
Let $I :=[0,1]$. Define $C_\lambda$ to be the set of continuous maps of $I$ which preserve the Lebesgue 
measure $\lambda$, and denote the measure-theoretic entropy of a map $f \in \Cl$ with respect to $\lambda$ by $h_{\lambda}(f)$.  We consider the  uniform topology.
The first result of this article is that there is a dense $G_\delta$  set of maps in $C_\lambda$ with zero Lebesgue measure-theoretic entropy, more precisely

\begin{theorem}\label{thm1}
The set $\{f \in C_{\lambda} : h_{\lambda}(f) = 0\}$  is a dense $G_\delta$ subset of $C_\lambda$.
\end{theorem}
Let us note that Bobok and Troubetzkoy \cite{BT} 
proved that $\{f \in C_{\lambda} : h_{\lambda}(f) = c\}$ is dense in $C_{\lambda}$, but only for each $c \in (0,\infty]$. 
The proof of Theorem \ref{thm1} follows a different path than the proofs of Parry \cite{Pa} or Friedland and Weiss \cite{FrWe}.
By conjugating, we see that the theorem also  holds more generally for $C_\mu$ where $\mu$ is any invariant non-atomic probability measure with full support \cite[Remark 1.1]{BCOT3}.

Adapting various known results and combining them with our results reveals sharp contrasts between the measure-theoretic and topological behavior of generic maps in $\Cl$.
In order to adapt the results of Bruckner and Garg \cite{BrGa},  we  recall a strong notion of nowhere differentiability: a map $f$ is of \emph{nonmonotic type} if  the upper derivative $\overline{D}f(x)=+\infty$ and the lower derivative $\underline{D}f(x)=-\infty$ for every $x$. 

Let $f\in C(I)$. For every $c\in \mathbb{R}$ the set $Lev_f(c) := \{x: f(x) = c\}$ is called a {\em  level set} of $f$. Following \cite[Definition 1.2]{BrGa} the level sets of a function $f\in  C(I)$ are said to be {\em normal} 
if there is a countable set $E_f \subset (\min f,\max f)$ such that the level set $Lev_f(c)$ of $f$ is
\begin{itemize}
\item[(i)] a non-empty perfect set when $c\notin E_f\cup\{\min f,\max f\}$,
\item[(ii)] a single point when $c=\min f$ or $c=\max f$,
\item[(iii)] the union of a nonempty perfect set $P$ with an isolated point $x\notin P$ when $c\in E_f$.
\end{itemize}

\begin{theorem}\label{thma}
The following two assertions hold.
\begin{enumerate}
    \item Generic  maps in $C_{\lambda}$ are of nonmonotonic type.
\item The level sets of  generic maps in $\Cl$ are normal.
\end{enumerate}
\end{theorem}

Clearly there are no continuous invertible interval maps that are nowhere differentiable.  
A map $f \in C_\lambda$ is called {\it invertible a.e.}~if there is an $f$-invariant set
$X \subset I$ of full Lebesgue measure such
that the restriction $f|_X$ is invertible.
Also let 
\begin{equation*}
B_0 := \{x \in I: \#Lev_f(x) > 1\}.
\end{equation*}
The following corollary contrasts the topological and measure-theoretic behavior of generic maps. 

\begin{corollary}\label{cor:ae}
There is a dense $G_\delta$ subset $Q$ of $C_\lambda$ such that:
\begin{enumerate}
    \item\label{cor:ae:1} For each  $f \in Q$ there exists a  set $X_f\subset I$ of Lebesgue measure $1$ such that $f$ is a bijection on $X_f$
and $\lambda(M)=\lambda(M\cap X_f)=\lambda(f^{-1}(M\cap X_f))$ for each Borel set $M\subset I$.

\item\label{cor:ae:2} For each map $f \in Q$ there exists a set $X_f\subset I$ of Lebesgue measure $1$ such that
 \begin{enumerate}
\item the level set $Lev_f(c)$ is a perfect set for each  $c\in (0,1)$, while 
\item for each $c\in I$ the set $X_f\cap\ Lev_f(c)$ contains at most one point.
\end{enumerate}
 
\item\label{cor:ae:3} Except for the two exceptional maps $\id$ and  $1-\id$ in $C_\lambda$ the
set $B_0$ contains an interval,  and thus we always have a set of positive Lebesgue measure on which $f$ is not one-to-one.\footnote{On the first page of the article \cite{BT} there is an incorrect statement that  
any $f \in \Cl$ other than $\id$ and $1-\id$ satisfies $h_{\lambda}(f) > 0$. This need not be the case as exhibited by the results of this article.}
\end{enumerate}
\end{corollary}

\begin{proof}
\eqref{cor:ae:1} According to \cite[Corollary 4.14.3]{Wal82} 
$h_{\lambda}(f)=0$ implies that $f$ is invertible a.e.. Thus combining this with Theorems \ref{thm1} and \ref{thma} yields statement \eqref{cor:ae:1}.  Statement \eqref{cor:ae:2}  follows immediately by combining \eqref{cor:ae:1} with 
Theorem  \ref{thma}.  Statement  \eqref{cor:ae:3}  follows from \cite[Corollary 3.2]{BCOT2} that except for the two exceptional maps $\id$ and $1-\id$ the
set $B_0$ contains an interval, and thus we always have a set of positive measure on which $f$ is not one-to-one.
\end{proof}

Unlike the approach given in \cite{Pa} and \cite{FrWe}, it is not clear if one can constructively define the set $X_f$ in Corollary
\ref{cor:ae} for a generic map in $\Cl$.

\begin{remark}
The reverse implication, $f$ is invertible a.e.\ implies that $h_\lambda(f)=0$, is not true in general. It is known that
any compact manifold of dimension $d > 1$ which carries a minimal uniquely ergodic homeomorphism also carries a minimal uniquely ergodic homeomorphism with positive topological entropy \cite{BCL} and the measure can be assumed
to be the Lebesgue measure.
It seems unknown if this reverse implication holds in $C_\lambda$ - compare with \cite[Theorem 1.9.7]{PrUr} and \cite[Corollary 4.18.1]{Wal82}.
\end{remark}

For comparison we recall other generic properties of maps from $C_\lambda$; in our context the item \eqref{item6} in the following theorem is particularly interesting. For the definitions of the notions we will mention in the following theorem we refer the reader to the associated articles.

\begin{theorem*} There is a dense $G_\delta$ subset $G$ of $C_\lambda$ such that for each $f \in G$ 
\begin{enumerate}
    \item   $f$ has infinite topological entropy \cite{BT},
    \item  $f$ is weakly mixing with respect to $\lambda$ \cite{BT},
    \item  $f$ is leo and thus satisfies the specification property \cite{BT},
    \item $f$ has the shadowing property \cite{BCOT1},
    \item periodic points have Cantor set structure, see \cite{BCOT1} for details.
    \item \label{item6} $f$ has $\lambda$-a.e.\ point a knot point \cite{B},
    \item  the graph of $f$ has  Hausdorff dimension and lower box dimension 1 and upper box dimension 2 \cite{BT},\cite{SW95}.
    \item  for every $\delta>0$ there is $n$ such that $f^n$ is $\delta$-crooked.
    \cite{CO}.
\end{enumerate}
\end{theorem*}

The leo and specification properties are actually satisfied not only generically but on an open dense subset of $C_\lambda$ \cite{BCOT3}.

Two other related results were proven by Sigmund. In his first result he considered the full shift
on $\sigma: \R^\Z \to \R^\Z$ and showed that the set of probability measures with zero entropy
form a dense $G_\delta$ in the weak topology.
He also considered  the set of all subshifts with a fixed finite alphabet with respect to the 
Vietoris topology and showed that the generic subshift has zero topological entropy \cite{Si}.

A generalization of Parry's result on genericity of zero entropy measures  
was obtained by Carvalho and Condori \cite{CC}. They
considered uniformly continuous maps $f$ defined over a Polish metric space and showed that the set 
of $f$-invariant measures with zero measure-theoretic entropy is a $G_\delta$ set in the weak topology and that this set is dense and thus generic if the set of $f$-periodic measures is dense in the set of $f$-invariant measures.

Let us give an outline of the article. In Section~\ref{sec:ProofThm1} we prove Theorem~\ref{thm1}. 
In Section~\ref{sec:Examples} we give a construction of Lebesgue measure-preserving maps
that have measure-theoretic entropy $0$ with respect to the Lebesgue measure. In Section~\ref{sec:ProofThm2} we prove Theorem~\ref{thma}, which is the second main result of our article. 
In Section~\ref{sec:planarattractors} we provide an application of our main results.
Part {\it(8)} from the previous theorem implies that inverse limit of any generic map from $C_{\lambda}$ is a curious topological object called the pseudo-arc.
We apply Theorem~\ref{thm1} and Corollary~\ref{cor:ae} to give additional information on some measure-theoretic aspects of the construction of planar homeomorphisms with pseudo-arc attractors from \cite{CO}.

\section{Proof of Theorem~\ref{thm1}}\label{sec:ProofThm1}

Let $\P := \{P_1,\dots,P_k\}$ be a finite partition of $I$ and let
\begin{equation*}
H(\P) := - \sum_{i=1}^k \lambda(P_i) \log \lambda(P_i).
\end{equation*}
For $f \in C_\lambda$ we denote
\begin{equation*}
h_n(f,\P) := \frac{1}{n} H( \bigvee_{i=0}^{n-1} f^{-i} \P).
\end{equation*}
This
function is decreasing in $n$ thus the  limit
$h(f,\P) := \lim_{n \to \infty} h_n(f,\P)$ exists.
Finally we define the {\it entropy} of $f$ by 
\begin{equation*}
h_\lambda(f) := \sup \{h(f,\P): \P \text{ is a
finite partition of } I\}.
\end{equation*}

Let $\P_i$ be the partition of $I$ into dyadic intervals
$[\frac{j}{2^i},\frac{j+1}{2^i}]$ and $\mathcal A_i$ the sub-algebra generated by $\P_i$.
Then $\mathcal A_1\subseteq \mathcal A_2\subseteq\cdots$ and $\bigvee_{n=1}^{\infty}\mathcal
A_n\mathring{=}~\mathcal B$ hence by \cite[Theorem 4.22]{Wal82}
$h_\lambda(f)  = \lim_{i \to \infty} h(f,\P_i).$  In what follows, we denote by $B(f,\delta)$ the  ball in $\Cl$ with the center at $f$ and of radius $\delta$ with respect to the uniform metric.

\begin{lemma}\label{l:1}Fix $f\in C_{\lambda}$ and choose an arbitrary interval $J=[a,b]\subset I$.
Then
\begin{equation*}
\forall~\eps>0~\exists~\delta>0 \text{ such that } \lambda(f^{-1}(J)\bigtriangleup g^{-1}(J))<\eps \text{ for all } g\in
B(f,\delta).
\end{equation*}
\end{lemma}
\begin{proof}Fix $\eps>0$ and let $\delta<\frac{\eps}{4}$. If we put
$H_{\delta}=f^{-1}([a+\delta,b-\delta])$, since $f\in C_{\lambda}$,
\begin{equation}\label{e:2}
\lambda(f^{-1}(J))=\lambda(J)>\lambda(H_{\delta})>\lambda(J)-\frac{\eps}{2}.\end{equation}
For each $g\in B(f,\delta)$ we can write analogously
\begin{equation}\label{e:3'}
\lambda(g^{-1}(J))=\lambda(J)>\lambda(H_{\delta})>\lambda(J)-\frac{\eps}{2}.\end{equation}
Clearly $H_{\delta}\subset (f^{-1}(J)\cap g^{-1}(J))$, so (\ref{e:2}) and
(\ref{e:3'}) imply
\begin{equation*}\lambda(f^{-1}(J)\bigtriangleup g^{-1}(J))<\eps. \qedhere
\end{equation*}
\end{proof}

\begin{lemma}\label{l:2}For $\ell\in\N$ let positive numbers $a_1,\dots,a_{\ell}$ satisfy
$\sum_{i=1}^{\ell}a_i=\eta<1$. Then
\begin{equation*}
-\sum_{j=1}^{\ell}a_j\log a_j\le \eta\log\ell - \eta\log\eta.
\end{equation*}
\end{lemma}
\begin{proof}
Let $b_i := a_i/\eta$, then $\sum_{i=1}^{\ell} b_i = 1$ and thus $-\sum_{i=1}^{\ell} b_i \log b_i \le \log \ell$
(\cite[Cor 4.2.1]{Wal82})
which is equivalent to the announced inequality.
\end{proof}
\begin{proposition}\label{p:1}Let $\P=\{P_1,\dots,P_k\}$ be a partition  of $I$ consisting of
intervals and $n\in\N$. The map 
\begin{equation*}
f\longrightarrow h_n(f,\P),~f\in C_{\lambda}
\end{equation*}
is continuous.
\end{proposition}
\begin{proof}Fix $\eps>0$ and $f\in C_{\lambda}$. Put
$\bigvee_{i=0}^{n-1}f^{-i}(\P)=\{A_1,\dots,A_m\}$, where $\lambda(A_j)>0$ for each
$j\in\{1,\dots,m\}$ and $\sum_{j=1}^m\lambda(A_j)=1$. 
Each set $A_j$ can be written as
\begin{equation*}
A_j=P_{i_0(j)}\cap f^{-1}(P_{i_1(j)})\cap\cdots\cap
f^{-(n-1)}(P_{i_{n-1}(j)}),~j\in\{1,\dots,m\},
\end{equation*}
and all other $(k^n-m)$ possible intersections have measure 0.

For $g$ close to $f$ we consider $\bigvee_{i=0}^{n-1} g^{-i} (\P)$ and let $\ell_g$ be the number
of positive measure elements of this partition. Enumerate the sets $B_{j,g}$ in $\bigvee_{i=0}^{n-1} g^{-i} (\P)$ in such a way that 
\begin{equation*}
B_{j,g}=P_{i_0(j)}\cap g^{-1}(P_{i_1(j)})\cap\cdots\cap
g^{-(n-1)}(P_{i_{n-1}(j)}),~j\in\{1,\dots,m\}.
\end{equation*}
It is clear that $B_{j,f} = A_j$.
Furthermore these sets vary continuously with respect to $g$ sufficiently close to $f$ in the
following sense.
Applying Lemma \ref{l:1} to the maps
$f^0,f^1,\dots,f^{n-1}\in C_{\lambda}$  
we can choose sufficiently small $\delta > 0$ so that each $B_{j,g}$ has positive measure provided that $A_j$ has positive measure for $1\leq j\leq m$. Hence $\ell_g \ge m$, we may assume that $\lambda(B_{j,g})>0$ for $1 \leq j\leq \ell_g$ and  
furthermore (as a consequence of Lemma~\ref{l:1}; see also Lemma~\ref{l:2}).

\begin{equation*}
-\sum_{j=1}^m\lambda(A_j)\log\lambda(A_j)+\sum_{j=1}^m\lambda(B_{j,g})\log\lambda(B_{j,g})<\frac{n\eps}{2},
\end{equation*}
 and
$\sum_{j=m+1}^{\ell_g}\lambda(B_{j,g})=1-\sum_{j=1}^{m}\lambda(B_{j,g})=\eta$ for which
\begin{equation*}
\eta\log(k^n-m)-\eta\log\eta<\frac{n\eps}{2}
\end{equation*}
for each $g\in
B(f,\delta)$. Since $\ell_g\le k^n$, from Lemma
\ref{l:2} we obtain
\begin{equation}\label{e:8}0<-\sum_{j=m+1}^{\ell_g}\lambda(B_{j,g})\log\lambda(B_{j,g})\le
\eta\log(k^n-m) - \eta\log\eta<\frac{n\eps}{2}.
\end{equation}
From  (\ref{e:8}) we obtain
\begin{equation*}
\left \vert
h_n(f,\P)-h_n(g,\P) \right \vert=\left \vert\frac{1}{n}H(\bigvee_{i=0}^{n-1}f^{-i}(\P))-\frac{1}{n}H(\bigvee_{i=0}^{n-1}g^{-i}(\P)) \right \vert<\eps.
\end{equation*}
\end{proof}

\begin{proof}[Proof of Theorem \ref{thm1}]
Fix $\beta > 0$ and let 
\begin{equation*}
Q_{\beta} := \bigcap_{k \ge 1} \bigcup_{i \ge k} \bigcup_{n \ge k} \left \{f \in C_\lambda:
h_n(f,\P_i) < \beta \right \}.
\end{equation*}

We claim that 
\begin{itemize}
\item[(i)] $Q_{\beta}$ a dense $G_\delta$ set and  
\item[(ii)] $Q_{\beta}$ coincides with the
set of maps
with entropy less than $\beta$.
\end{itemize}
Once the claim is established the theorem follows since $Q := \bigcap_{n \ge 1} Q_{1/n}$
is a dense $G_\delta$ set and it coincides with the set of zero entropy maps.

We turn to the proof of the claim. Let us first prove (ii).
If $h_\lambda(f) < \beta$ then $h_n(f,\P_i) < \beta$ for each $i$ for all sufficiently large $n$,
so $f \in Q_{\beta}$. Conversely  to see that the entropy of the maps in $Q_{\beta}$ is
at most $\beta$ we
fix $f \in Q_{\beta}$. By the definition of $Q_{\beta}$, for each $f \in Q_{\beta}$ there are sequences
$i_k \ge k$ and $n_k \ge k$ such that
$h_{n_k} (f,\P_{i_k}) < \beta$ for each $k \ge 1$.
Since $h_n(f,\P)$ is decreasing in $n$ we  conclude $h(f,\P_{i_k}) < \beta$.

Since $\P_{i+1}$ refines $\P_i$ we have $h(f,\P_i)$ is monotonic in $i$ and we conclude $h(f,\P_i) < \beta$
for all $i \geq i_1$. Thus $h_\lambda(f)  = \lim_i h(f,\P_i) < \beta$ which finishes the proof of (ii).

To prove (i) we remark that the set $\{f \in C_\lambda: h_n(f,\P_i) < \beta \}$ is open by
Proposition \ref{p:1}, thus $Q_{\beta}$ is a $G_\delta$ set.

The set $Q_{\beta}$ coincides with the set of maps with entropy less than
$\beta$, thus its density follows from  \cite{BT} Proposition 24. This
finishes the proof of the Theorem.
 \end{proof}

 \section{Constructions of zero entropy maps}\label{sec:Examples}

Suppose $f\in C_\lambda$ be such that $\lambda(\Per(f))=1$.
If we use the ergodic decomposition theorem \cite[p. 153]{Wal82} then
\begin{equation}\label{e:4}
h_{\lambda}(f)=
\int_{E(I,f)}h_{\mu}(f)~d\nu(\mu)
\end{equation}
where $E(I,f)$ is the set of all ergodic measures invariant for $f$ and $\nu$ is a Borel measure on $E(I,f)$ satisfying 
\begin{equation*}
\lambda=\int_{E(I,f)}\mu~ d\nu(\mu).
\end{equation*}
Since $\lambda(\Per(f))=1$ necessarily $\nu$ a.e.\ measure ergodic $\mu$ is a CO-measure (i.e., a measure on a periodic orbit).
In \cite[Theorem 2]{BCOT1} we constructed a dense set of leo maps for which the periodic points have full measure,
and thus they all have zero entropy with respect to Lebesgue measure.
This yields an alternative proof of the density of  $Q$ from the proof of Theorem~\ref{thm1}.

\begin{remark} Parry sketched a construction of maps $f \in C_\lambda$ which are  piecewise monotone with $2$ full branches which are invertible $\lambda$-a.e.\ \cite{Pa}; a more detailed version of this construction (with $k$ full branches) was given by Friedland and Weiss \cite{FrWe}.  
\end{remark}

\section{Proof of Theorem~\ref{thma}}\label{sec:ProofThm2}

Let $C(I)$ denote the set of continuous maps $f: I \to I$.
 A function $f\in C(I)$ is said to be {\em non-decreasing at a point $x\in I$} if there exists a $\delta>0$ such that $f(t)\leq f(x)$ when  $t\in I\cap (x-\delta,x)$ and $f(t)\geq f(x)$  when $t\in I\cap (x,x-\delta)$. The function $f$ is called {\em nonincreasing  at $x$} if $-f$ is nondecreasing at $x$, and $f$ is called {\em monotone at $x$} if it is either nondecreasing or nonincreasing at $x$. If $f\in C(I)$ and $\gamma\in \mathbb{R}$ then we define the map $f_{-\gamma}$ by:
\begin{equation}\label{e:f-gamma}
f_{-\gamma}(x) :=f(x)-\gamma x \text{ for } x\in I.
\end{equation}
A map $f$ is {\em monotonic type at $x$} if there exists a real number $\gamma$ such that the function $f_{-\gamma}$ is monotone at $x$. If  $f$ is not of monotonic type at any point of $I$, it will be said to be of {\em nonmonotonic type}. It is easy to see that $f$ is of nonmonotonic type if and only if for each $x$, $\overline{D}f(x)=+\infty$ and $\underline{D}f(x)=-\infty$, where $\overline{D}f$, $\underline{D}f$ are the upper and lower derivatives.

 The proof of Theorem \ref{thma} requires two auxiliary lemmas.
 The first lemma is a minor modification of Lemma 11 from \cite{BCOT4}.

\begin{lemma}
Let $f$ be a piecewise affine continuous interval map with nonzero slopes and such that its derivative exists everywhere but on a finite set $E$. 
Then $f\in C_{\lambda}$ if and only if 
\begin{equation}\label{e:3}
   \forall~y\in [0,1)\setminus F(E)\colon~\sum_{x\in F^{-1}(y)}\frac{1}{\vert F'(x)\vert}=1.
 \end{equation}
 \end{lemma}

\begin{figure}
	\begin{tikzpicture}[scale=5]
	\draw(0,0)--(0,1)--(1,1)--(1,0)--(0,0);
	\draw[dotted,blue] (0,1/3)--(1,1/3);
	\draw[dotted,blue] (0,2/3)--(1,2/3);
	\draw[dotted,blue] (0,5/6)--(1,5/6);
	\draw[dotted,blue] (0,1/6)--(1,1/6);
	\draw (0,1/6)--(1/18,1/3)--(1/18+1/15,2/3)--(1/18+1/15+1/24,5/6)--(1/18+1/15+1/24+1/12-1/19,1)--(1/18+1/15+1/24+1/12+1/12,5/6)--(1/18+1/15+1/24+1/12+1/12+1/24,2/3)--(1/18+1/15+1/24+1/12+1/12+1/24+1/15,1/3)--(1/18+1/15+1/24+1/12+1/12+1/24+1/18+1/15,1/6)--(1/18+1/15+1/24+1/12+1/12+1/24+1/18+1/15+1/12+1/14,0)--(1/18+1/15+1/24+1/12+1/12+1/24+1/18+1/15+1/12+1/12,1/6)--(1/18+1/15+1/24+1/12+1/12+1/24+1/18+1/15+1/12+1/12+1/18,1/3)--(1/18+1/15+1/24+1/12+1/12+1/24+1/18+1/15+1/12+1/12+1/18+1/15,2/3)--(1/18+1/15+1/24+1/12+1/12+1/24+1/18+1/15+1/12+1/12+1/18+1/15+1/24+1/40,5/6)--(1/18+1/15+1/24+1/12+1/12+1/24+1/18+1/15+1/12+1/12+1/18+1/15+1/24+1/24,2/3)--(1/18+1/15+1/24+1/12+1/12+1/24+1/18+1/15+1/12+1/12+1/18+1/15+1/24+1/12+1/15,1/3)--(1/18+1/15+1/24+1/12+1/12+1/24+1/18+1/15+1/12+1/12+1/18+1/15+1/24+1/24+2/15,2/3);
	\node[blue] at (-0.05,1/6){\tiny $c_1$};
	\node[blue] at (-0.05,1/3){\tiny $c_2$};
	\node[blue] at (-0.05,2/3){\tiny $c_3$};
	\node[blue] at (-0.05,5/6){\tiny $c_4$};
	\node[blue] at (0, -0.08){\tiny $0$};
	\node[blue] at (1, -0.08){\tiny $1$};
	\node[blue] at (1/2+1/16,1/2){\tiny $f_n\in C_{\lambda}$};
	\end{tikzpicture}\nolinebreak
	\hspace{1cm}
	\begin{tikzpicture}[scale=5]
	\node[blue] at (3/4,1/2){\tiny $g_{n,i}\in C_{\lambda}$};
	\draw(0,0)--(0,1)--(1,1)--(1,0)--(0,0);
	\draw[dotted,blue] (0,3/8)--(3/8,3/8);
        \draw[dotted,blue] (3/8-1/16,0)--(3/8-1/16,1/4);
	\draw[dotted,blue] (0,1/4)--(3/8,1/4);
	\draw[dashed,blue] (1/4,1/4)--(1/4,3/8)--(3/8,3/8)--(3/8,1/4)--(1/4,1/4);
	\draw (0,0)--(1/4,1/4);
	\draw (3/8,3/8)--(1,1);
	\draw(1/4,1/4)--(1/4+1/32,3/8)--(1/4+1/16,1/4)--(3/8,3/8);
	\draw[blue,dotted] (1/4,1/4)--(3/8,3/8);
	\node[blue] at (-0.05,3/8){\tiny $\frac{k+1}{n}$};
	\node[blue] at (-0.05,1/4){\tiny $\frac{k}{n}$};
	\node[blue] at (1/4,-0.05){\tiny $\frac{k}{n}$};
	\node[blue] at (3/8+1/20,-0.05){\tiny $\frac{k+1}{n}$};
	\node[blue] at (5/16,-0.05){\tiny $c_i$};
	\draw[dotted,blue] (3/8,0)--(3/8,3/8);
	\draw[dotted,blue] (1/4,0)--(1/4,1/4);
	\draw[dotted,blue] (3/8,0)--(3/8,1/4);
	\end{tikzpicture}
 \caption{Removing irrational critical values.}\label{fig1}
\end{figure}
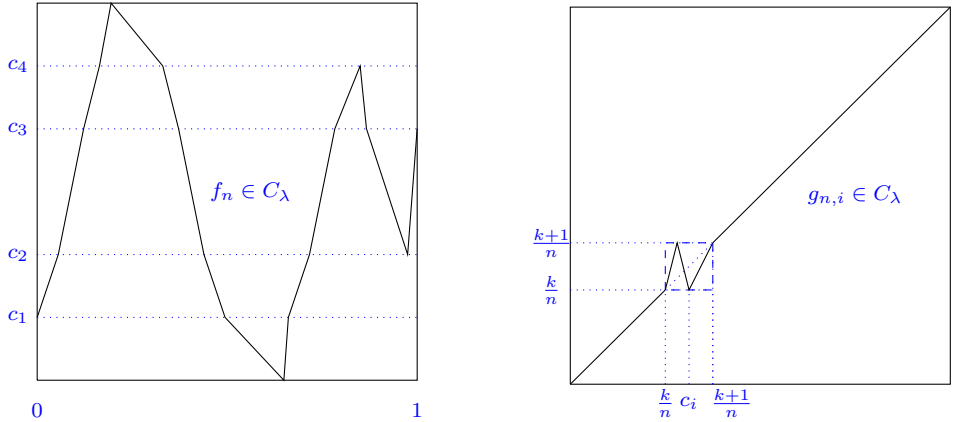

A {\em critical value} of a map $f$ is the value of $f$ at a local maximum or a local minimum. A {\em determining 
value} of a piecewise affine map $f$ is either one of its critical values or its value at a point in which 
the derivative of $f$ does not exist. 

\begin{lemma}\label{crv}
The set of piecewise affine maps with rational determining values is dense in $C_{\lambda}$.
 \end{lemma}
\begin{proof} Consider a dense collection of piecewise affine maps $\{f_m\} \subset \Cl$ (Figure \ref{fig1} left), each $f_m$ has finitely
many  determining points. 
Fix $m$ and choose $n = n(m)$  so large that 
each of the intervals $[k/n,k+1/n]$ contains at most one  determining value.
For each irrational  determining value $c_i$   of $f_m$ we consider the map $g_{n,i} \circ f_m$, where $g_{n,i}$ is the map of Figure \ref{fig1} right.  
The resulting map has
 the  determining value  $c_i$ replaced by $k/n$  and $(k+1)/n$.  We progressively remove all irrational  determining
 values of $f_m$ to obtain a map $h_m$ with only rational  determining values. 
 The set $\{h_m\}$ is dense in $\Cl$ since $\rho(f_m,h_m) \to 0$. 
 \end{proof}

The proof of Theorem \ref{thma} follows the strategy of the proof of \cite[Theorem 2.2]{BrGa}.  
We use the terminology of the proof from \cite{BrGa}. 
In particular, for any $n \in \N$, the set $A_n$  denotes the set of functions $f \in \Cl$
for which there exist $\gamma \in [-n,n]$ and $x \in [0,1]$ such that $f_{-\gamma}(t) \le f_{-\gamma}(x)$
when $t \in [0,1] \cap (x-1/n,x)$ and $f_{-\gamma}(t) \ge f_{-\gamma}(x)$
when $t \in [0,1] \cap (x,x + 1/n)$. 
Remember that $f_{-\gamma}$ was defined 
in \eqref{e:f-gamma}, here we changed the notation of \cite{BrGa} replacing their $\lambda$ by $\gamma$ to avoid a collision of terminologies.

\begin{proof}[Proof of Theorem \ref{thma}]
Except for one step, the strategy of the proof  is the same as in \cite{BrGa}; this difference is due to the fact of the preserving of Lebesgue measure.  The step which differs is showing
that each $A_n$ is nowhere dense,  this step is the most technical step in the proof in \cite{BrGa}. 
To show this we proceed as illustrated in Figure \ref{fig2}. We know that piecewise affine maps are dense in $C_{\lambda}$ \cite{BT}. So from now $n\in \mathbb N$ 
defining $A_n$ is fixed.

{\bf I.~A choice of $f$.} Using Lemma \ref{crv}, 
given an open set $U\subset C_{\lambda}$ and $\eps>0$ 
we can choose a piecewise affine map $f\in U$ whose determining values  
are rational such that  $B(f,\eps)\subset U$.
Let $q$ be the greatest common denominator of the 
determining values. Let  $1/\tau$ be a positive integer multiple of $q$ such that $\tau < \eps /3$.
For each integer $1 \le i \le  1/\tau $ consider the interval $J_i = [(i-1)\tau,i\tau]$.
These intervals satisfy

\begin{enumerate}[label*=(\roman*)]
\item $\lambda(J_i) = \tau$ for all $i$, 
\item $J^{\circ}_i\cap J^{\circ}_j=\emptyset$ for all $i \ne j$, 
\item each $J_i$ has no determining value in its interior, 
\item for each $i$, all points inside a $J_i$ have the same number $n(i)$ of preimages,
\item $f^{-1}(J_i^{\circ})$ has $n(i)$ connected components $[u_i^{k},v_i^k]$, $1\le k\le n(i)$,
\item $f([u_i^{k},v_i^k])=J_i$ for each $k$,
\item \label{inside} every window perturbation of $f$ on $[u_i^{k},v_i^k]$ remains in\\
$B(f,\tau) \subset  B(f,\eps/3)$; furthermore  this containment holds if we make a finite number of such window perturbations (see section 2.1 of \cite{BCOT3} for the definition of window perturbations).
\end{enumerate}

{\bf II.\ A perturbation $\tilde f$ of $f$.} 
We choose $\delta>0$ so small that 
\begin{equation}\nonumber\label{choicetau}  \delta < \text{min}_{i,k} \left (  (v_i^k - u^k_i)/2 \right )
\end{equation} 
and satsifying a further assumption which will be made in the next step.

If $x$ a common point of two adjacent $[u_i^k,v_i^k]$ intervals then we call the intervals 
$[x-\delta,x]$ and $[x,x+\delta]$ {\em $\delta$-marked intervals.}   
Each $\delta$-marked interval is contained in a single interval $[u_i^k,v_i^k]$ and its image
is contained in a single $J_i$ and $\delta$-marked intervals have pairwise disjoint interiors.
 We consider a (not necessarily regular) piecewise affine window perturbation $\tilde f\in B(f,\tau)$ of $f$ on all $[u_i^{k},v_i^k]$'s such that
 (see Figure \ref{fig2}, left)
\begin{enumerate}[label=(\roman*)]
\setcounter{enumi}{7}
\item
if an interval  $K \subset [u_i^{k},v_i^k]$ and $\lambda(K)=\delta$ then $K$ contains either at least three critical points of $\tilde f$ or two critical points of $\tilde f$ and one of the endpoints $u_i^{k},v_i^k$,\label{slope}
\item
if $[u_i^{1},v_i^1]$ is the leftmost interval and $f$ is increasing on $[u_i^{1},v_i^1]$ then $\tilde f(u_i^1)=f(v_i^1)$, $\tilde f(v_i^1)=f(v_i^1)$, the map $\tilde f$ has two laps on the $\delta$-marked interval $[v_i^{1}-\delta,v_i^1]$,
\item
if $[u_i^{n(i)},v_i^{n(i)}]$ is the rightmost interval and $f$ is increasing on $[u_i^{n(i)},v_i^{n(i)}]$ then $\tilde f(u_i^{n(i)})=f(u_i^{n(i)})$, $\tilde f(v_i^{n(i)})=f(u_i^{n(i)})$, the map $\tilde f$ has two laps on the $\delta$-marked interval $[u_i^{n(i)},u_i^{n(i)}+\delta]$,
\item
otherwise $\tilde f(u_i^k)=f(u_i^k)$ and $\tilde f(v_i^k)=f(v_i^k)$ for each $i,k$, the map $\tilde f$ has two laps on the  $\delta$-marked intervals $[u_i^k,u_i^k+\delta]$ and $[v_i^{k}-\delta,v_i^k]$.
\end{enumerate}
By construction, the map $\tilde f \in B(f,\tau) \subset B(f,\eps/3)$ fulfills (\ref{e:3}), so $\tilde f\in C_{\lambda}$. 

 \begin{figure}
\begin{minipage}{.5\textwidth}
	\begin{tikzpicture}[scale=6]
	\draw[dotted] (1/8,0)--(1/8,1);
	\draw[dotted] (1/4,0)--(1/4,1);
	\draw[dotted,blue] (0,1/3)--(1,1/3);
	\draw[dotted,blue] (0,2/3)--(1,2/3);
	\draw[dotted,blue] (2/3,0)--(2/3,1);
	\draw[dotted,blue] (5/6,0)--(5/6,1);
	\draw[dotted,blue] (11/24,0)--(11/24,1);
	\draw(0,0)--(0,1)--(1,1)--(1,0)--(0,0);
	\draw(0,2/3)--(1/64,1/3)--(2/64,2/3)--(3/64,1/3)--(4/64,2/3)--(5/64,1/3)--(6/64,2/3);
	\draw[red](6/64,2/3)--(7/64,1/3)--(8/64,2/3)--(8/56,1)--(9/56,2/3);
	\node[circle,fill, inner sep=1] at (1/8,2/3){};
	\node[circle,fill, inner sep=1] at (1/4,1){};
	\node[circle,fill, inner sep=1] at (11/24,2/3){};
	\node[circle,fill, inner sep=1] at (2/3,1/3){};
	\node[circle,fill, inner sep=1] at (5/6,0){};
	\node[circle,fill, inner sep=1,red] at (111/264,2/3){};
	\node[circle,fill, inner sep=1,red] at (131/264,2/3){};
	\node[circle,fill, inner sep=1,red] at (166/264,1/3){};
	\node[circle,fill, inner sep=1,red] at (186/264,1/3){};
	\node[circle,fill, inner sep=1,red] at (6/64,2/3){};
	\node[circle,fill, inner sep=1,red] at (9/56,2/3){};
	\node[circle,fill, inner sep=1,red] at (111/264,0){};
	\node[circle,fill, inner sep=1,red] at (131/264,0){};
	\node[circle,fill, inner sep=1,red] at (166/264,0){};
	\node[circle,fill, inner sep=1,red] at (186/264,0){};
	\node[circle,fill, inner sep=1,red] at (6/64,0){};
	\node[circle,fill, inner sep=1,red] at (9/56,0){};
	\node[circle,fill, inner sep=1] at (0,2/3){};
	\draw[red,thick] (111/264,0)--(131/264,0);
	\draw[red,thick] (166/264,0)--(186/264,0);
	\draw[red,thick] (6/64,0)--(9/56,0);
		\node[circle,fill, inner sep=1] at (1,0){};
			\node[blue] at (0, -0.05){\tiny $0$};
				\node[blue] at (1, -0.05){\tiny $1$};
					\node[blue] at (-0.05,1/3){\tiny $\frac{1}{3}$};
						\node[blue] at (-0.05,2/3){\tiny $\frac{2}{3}$};
							\node[red] at (1/8, 0.04){\small $2\delta$};
								\node[blue] at (1/16, -0.1) {\tiny $[u^1_2,v^1_2]$};
									\node[blue] at (0.291, -0.1){ \tiny $[u^1_3,v^1_3]$};
												\node[blue] at (9/16, -0.1){ \tiny $[u^2_2,v^2_2]$};
															\node[blue] at (5/6, -0.1){ \tiny $[u^1_1,v^1_1]$};
															\draw[blue](0,0)--(1/16,-0.05)--(1/8,0);
															\draw[blue](1/8,0)--(0.291,-0.05)--(11/24,0);
															\draw[blue](11/24,0)--(9/16,-0.05)--(2/3,0);
															\draw[blue](2/3,0)--(5/6,-0.05)--(1,0);
																\node[blue] at (-0.05,1/6){\tiny $J_1$};
																	\node[blue] at (-0.05,1/2){\tiny $J_2$};
																		\node[blue] at (-0.05,5/6){\tiny $J_3$};
	
	\draw(9/56,2/3)--(10/56,1)--(11/56,2/3)--(12/56,1)--(13/56,2/3)--(14/56,1);
	
	\draw(14/56,1)--(71/264,2/3)--(76/264,1)--(81/264,2/3)--(86/264,1)--(91/264,2/3)--(96/264,1)--(101/264,2/3)--(106/264,1)--(111/264,2/3);
	\draw[red](111/264,2/3)--(116/264,1)--(126/264,1/3)--(131/264,2/3);
	\draw (131/264,2/3)--(136/264,1/3)--(141/264,2/3)--(146/264,1/3)--(151/264,2/3)--(156/264,1/3)--(161/264,2/3)--(166/264,1/3);
	\draw[red](166/264,1/3)--(171/264,2/3)--(176/264,1/3)--(181/264,0)--(186/264,1/3);

	\draw(186/264,1/3)--(31/44+17/1188,0)--(31/44+34/1188,1/3)--(31/44+51/1188,0)--(31/44+68/1188,1/3)--(31/44+85/1188,0)--(31/44+103/1188,1/3)--(31/44+120/1188,0)--(31/44+137/1188,1/3)--(31/44+154/1188,0)--(31/44+154/1188+1/48,1/3)--(31/44+154/1188+2/48,0)--(31/44+154/1188+3/48,1/3)--(31/44+154/1188+4/48,0)--(31/44+154/1188+5/48,1/3)--(31/44+154/1188+6/48,0)--(31/44+154/1188+7/48,1/3)--(1,0);

	\draw[blue] (0,1/3)--(1/4,1)--(2/3,1/3)--(5/6,0)--(1,1/3);
	\end{tikzpicture}
\end{minipage}\nolinebreak
\begin{minipage}{.5\textwidth}
	\begin{tikzpicture}[scale=6]
	\draw[dotted] (1/8,0)--(1/8,1);
	\draw[dotted] (1/4,0)--(1/4,1);
	\draw[dotted,blue] (0,1/3)--(1,1/3);
	\draw[dotted,blue] (0,2/3)--(1,2/3);
	\draw[dotted,blue] (2/3,0)--(2/3,1);
	\draw[dotted,blue] (5/6,0)--(5/6,1);
	\draw[dotted,blue] (11/24,0)--(11/24,1);
	\draw(0,0)--(0,1)--(1,1)--(1,0)--(0,0);
	\draw(0,2/3)--(1/64,1/3)--(2/64,2/3)--(3/64,1/3)--(4/64,2/3)--(5/64,1/3)--(6/64,2/3);
	\draw[red](6/64,2/3)--(7/64,1)--(8/64,2/3)--(8/56,1/3)--(9/56,2/3);
	\node[circle,fill, inner sep=1] at (1/8,2/3){};
	\node[circle,fill, inner sep=1] at (1/4,1){};
	\node[circle,fill, inner sep=1] at (11/24,2/3){};
	\node[circle,fill, inner sep=1] at (2/3,1/3){};
	\node[circle,fill, inner sep=1] at (5/6,0){};
	\node[circle,fill, inner sep=1,red] at (111/264,2/3){};
	\node[circle,fill, inner sep=1,red] at (131/264,2/3){};
	\node[circle,fill, inner sep=1,red] at (166/264,1/3){};
	\node[circle,fill, inner sep=1,red] at (186/264,1/3){};
	\node[circle,fill, inner sep=1,red] at (6/64,2/3){};
	\node[circle,fill, inner sep=1,red] at (9/56,2/3){};
	\node[circle,fill, inner sep=1,red] at (111/264,0){};
	\node[circle,fill, inner sep=1,red] at (131/264,0){};
	\node[circle,fill, inner sep=1,red] at (166/264,0){};
	\node[circle,fill, inner sep=1,red] at (186/264,0){};
	\node[circle,fill, inner sep=1,red] at (6/64,0){};
	\node[circle,fill, inner sep=1,red] at (9/56,0){};
	\node[circle,fill, inner sep=1] at (0,2/3){};
	\draw[red,thick] (111/264,0)--(131/264,0);
	\draw[red,thick] (166/264,0)--(186/264,0);
	\draw[red,thick] (6/64,0)--(9/56,0);
		\node[circle,fill, inner sep=1] at (1,0){};
	
			\node[blue] at (0, -0.05){\tiny $0$};
				\node[blue] at (1, -0.05){\tiny $1$};
					\node[blue] at (-0.05,1/3){\tiny $\frac{1}{3}$};
						\node[blue] at (-0.05,2/3){\tiny $\frac{2}{3}$};
						
							\node[red] at (1/8, 0.04){\small $2\delta$};
								\node[blue] at (1/16, -0.1) {\tiny $[u^1_2,v^1_2]$};
									\node[blue] at (0.291, -0.1){ \tiny $[u^1_3,v^1_3]$};
												\node[blue] at (9/16, -0.1){ \tiny $[u^2_2,v^2_2]$};
															\node[blue] at (5/6, -0.1){ \tiny $[u^1_1,v^1_1]$};
															\draw[blue](0,0)--(1/16,-0.05)--(1/8,0);
															\draw[blue](1/8,0)--(0.291,-0.05)--(11/24,0);
															\draw[blue](11/24,0)--(9/16,-0.05)--(2/3,0);
															\draw[blue](2/3,0)--(5/6,-0.05)--(1,0);
																\node[blue] at (-0.05,1/6){\tiny $J_1$};
																	\node[blue] at (-0.05,1/2){\tiny $J_2$};
																		\node[blue] at (-0.05,5/6){\tiny $J_3$};
	
	\draw(9/56,2/3)--(10/56,1)--(11/56,2/3)--(12/56,1)--(13/56,2/3)--(14/56,1);
	
	\draw(14/56,1)--(71/264,2/3)--(76/264,1)--(81/264,2/3)--(86/264,1)--(91/264,2/3)--(96/264,1)--(101/264,2/3)--(106/264,1)--(111/264,2/3);
	\draw[red](111/264,2/3)--(116/264,1/3)--(126/264,1)--(131/264,2/3);
	\draw (131/264,2/3)--(136/264,1/3)--(141/264,2/3)--(146/264,1/3)--(151/264,2/3)--(156/264,1/3)--(161/264,2/3)--(166/264,1/3);
	\draw[red](166/264,1/3)--(171/264,0)--(176/264,1/3)--(181/264,2/3)--(186/264,1/3);
	
	\draw(186/264,1/3)--(31/44+17/1188,0)--(31/44+34/1188,1/3)--(31/44+51/1188,0)--(31/44+68/1188,1/3)--(31/44+85/1188,0)--(31/44+103/1188,1/3)--(31/44+120/1188,0)--(31/44+137/1188,1/3)--(31/44+154/1188,0)--(31/44+154/1188+1/48,1/3)--(31/44+154/1188+2/48,0)--(31/44+154/1188+3/48,1/3)--(31/44+154/1188+4/48,0)--(31/44+154/1188+5/48,1/3)--(31/44+154/1188+6/48,0)--(31/44+154/1188+7/48,1/3)--(1,0);
	
	\draw[blue] (0,1/3)--(1/4,1)--(2/3,1/3)--(5/6,0)--(1,1/3);
	\end{tikzpicture}
\end{minipage}
\caption{ Left: the map $\tilde f$ in step II. 
 Right: the map $F$ in step  III. }\label{fig2}
\end{figure}
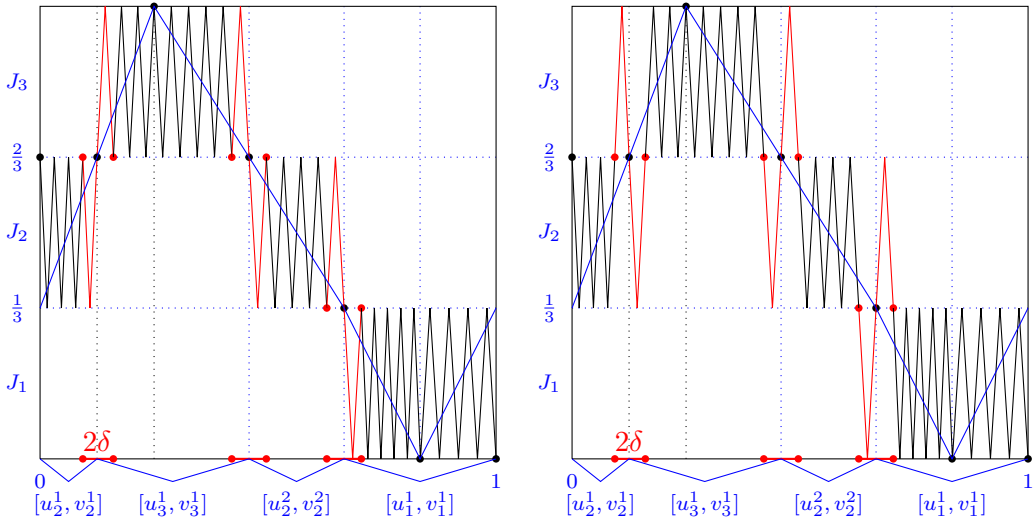

{\bf III. Constructing $F$ from $\tilde f$.} Finally we construct the map $F$, adapting the map $\tilde{f}$, such that
 (see Figure \ref{fig2}, right)
\begin{itemize}
\item[(xii)] the graphs of $\tilde f$ and $F$ coincide outside of the $\delta$-marked intervals,
\item[(xiii)]  if $J=[d,e]$ is a $\delta$-marked interval, then $\tilde f(d)=\tilde f(e)$; the graph of $F\vert_{J}$ is a  copy of the graph $\tilde f\vert_{J}$ reflected in the line $y=f(d)$,  i.e., $F(x) = 2 \tilde f(d) - \tilde f(x) $; equivalently if $J=[d,e]$ is the union of the two consecutive  $\delta$-marked intervals, then $\tilde f(d)=\tilde f(e)$; the graph of $F\vert_{J}$ is a  copy of the graph $\tilde f\vert_{J}$ reversed,  i.e., $F(x) = \tilde f(d+e - x)$.
\end{itemize}

By  construction, the map $F$ fulfills (\ref{e:3}), so $F\in C_{\lambda}$.
 By \ref{inside}, $F \in B(\tilde f, 2 \tau) \subset B(f,\eps)\subset U$.
 Let us show that $F\notin A_n$. Notice that by (\ref{choicetau}) and \ref{slope} the map $F$ has absolute value of the slope
greater than $10n$ on each affine part of its graph. 

If $x\in [0,1]$ and $F$ is decreasing on some one-sided neighborhood $V$ of $x$  then $F_{-\gamma}$, $\gamma\in [-n,n]$, has its derivative less than $-9n$ on the same neighborhood $V$ of $x$, so either there is $t\in (x,x+\frac{1}{n})$ with $F_{-\gamma}(x)>F_{-\gamma}(t)$ or $t\in (x-\frac{1}{n},x)$ for which $F_{-\gamma}(t)>F_{-\gamma}(x)$. This is the case when $x$ is an endpoint or a critical point of $F$ or such that $F$ is decreasing on some two-sided neighborhood of $x$.

It remains to consider the case when $x\in (0,1)$ and $F$ is increasing on the two-sided neighborhood of $x$. Let $x\in [u_i^k,v_i^k]$. We distinguish several possibilities. 

For convenience denote the endpoints of the interval $J_i$ by  $a_i = (i-1) \tau$ and $b_i = i \tau$.

We additionally assume that
\begin{equation*}
\delta < \frac{\tau}{10n}.
\end{equation*}
The first case is when $F(x)\in J_i$ and $F(x)\le (a_{i}+b_{i})/2$. In this case by \ref{slope} there has to be a point $y\in [u_i^k,x)$ for which $x-y<3\delta$ and $F(y)=b_i $. Then since 
\begin{equation*}\label{eineq}3\delta<\frac{3\tau}{10n}<\frac{3}{10n}<\frac1{n}
\quad \text{  and  } \quad 
\frac{F(x)-F(y)}{x-y}<-\frac{\tau}{6\delta}<-\frac{5n}{3},
\end{equation*}
we obtain 
\begin{equation*}\frac{F_{-\gamma}(x)-F_{-\gamma}(y)}{x-y}<-\frac{2n}{3}\text{ for }y\in (x-\frac{1}{n},x).
\end{equation*}

The second case is $F(x)\in J_i$ and $F(x)> (a_{i}+b_{i})/2$.  Again by \ref{slope} there has to be a point $y\in (x,v_i^k)$ for which $y-x<3\delta$ and $F(y)=a_i$. Analogously as before  
\begin{equation*}
\frac{F(x)-F(y)}{x-y}<-\frac{5n}{3},
\end{equation*}
hence  
\begin{equation*}\frac{F_{-\gamma}(x)-F_{-\gamma}(y)}{x-y}<-\frac{2n}{3}\text{ for }y\in (x,x+\frac{1}{n}).
\end{equation*}
 
The last case is when  $F(x)\in J_{\ell}
$ with $\ell\neq i$. Then 
 
 - if $\ell=i+1$, one can proceed analogously as above taking $y$ the closest critical point to the right of $x$  where the local minimum is attained,\\ 
 - if $\ell=i-1$, one can proceed analogously as above taking $y$ the closest critical point to the left of $x$ where the local maximum is attained.

This finishes the proof of (i).
The assertion (ii) on level sets follow using the same arguments 
as in Theorem 3.3 in \cite{BrGa}, but replacing their Theorem 2.2
by (i). Note that if $f\in C_{\lambda}$ then $\min f=0$ and $\max f=1$.
\end{proof}

\section{Consequences for a parametrized family of planar attractors}\label{sec:planarattractors}

In this section, we apply Theorem~\ref{thm1} and Corollary~\ref{cor:ae} to provide a better understanding of some planar attractors from \cite{CO}. 
We find these results surprising, because  they highlights the dichotomy between topological and measure-theoretic perception of these attractors.
Before we proceed we need some additional terminology.

Let $X$ be a compact metric space and let $F\colon X\to X$ be continuous.
Define
\begin{equation*}
\hat X_F:=\underleftarrow{\lim} (X,F)
=
\{\hat x:=\big(x_{0},x_1,x_2,\ldots \big) \in \Pi^{\infty}_{i=1} X : x_i=F(x_{i+1}), \forall i \geq 0\},
\end{equation*}
where $F$ is called the {\em bonding map}.
We equip the {\em inverse limit} $\hat X_F$ with the subspace
metric which is induced by
\emph{product metric} in $\Pi^{\infty}_{i=1} X_i$.
We can define the shift homeomorphism $\hat F: \hat X_F\to \hat X_F$  defined for $(x_{0},x_1,x_2,\ldots)\in \hat X_F$ by 
\begin{equation*} 
\hat F((x_{0},x_1,x_2,\ldots))=(F(x_{0}),F(x_1),F(x_2),\ldots)=(F(x_0),x_0,x_1,\ldots).
\end{equation*}

Denote by $\pi_n:\hat X_F\to X$ the coordinate projection maps. Recall also that $\mathcal{B}(X)$ denotes the $\sigma$-algebra of Borel sets in $X$. 

For the details of what follows we refer the reader to \cite{CO}.
Let $\mathcal{L}$ be the Lebesgue measure on a topological disk $\mathcal{D}$ and $\lambda$ the Lebesgue measure on $I=[0,1]$. Suppose that
$I\times\{0\}\subset \mathrm{int}(\mathcal{D})$ and let $f:I\to I$ be a continuous map. The map $f$ can be extended to a map $F$ defined as a uniform limit of homeomorphisms of $\mathcal{D}$. Let us denote 
$\hat{\mathcal{D}}:=\underleftarrow{\lim} (\mathcal{D},F)$; by Brown's theorem $\hat{\mathcal{D}}$ is again a topological disk.
Let $\mathcal{B}(\hat{\mathcal{D}})$ be the smallest $\sigma$-algebra on $\hat{\mathcal{D}}$ 
such that all the projection maps $\pi_i$ are measurable. Then there exists a unique probability measure $\hat{\mathcal{L}}$ on $\mathcal B(\hat D)$  such that $\hat{\mathcal{L}}(\pi^{-1}_n(A))=\mathcal{L}(A)$
 for all $A\in \mathcal B(\mathcal{D})$ and each $n\in \N_0$ where (the measure $\hat \lambda$ is defined in an analogous way).

Let us recall a result from \cite{Parth}.
Let $\mu$ be an $F$-invariant  Borel probability measure on $\mathcal{D}$. The set $B_{\mu}$ which consists of all points $x\in \mathcal{D}$ such that 
\begin{equation*}
\lim_{n\to \infty} \frac{1}{n}\sum_{i=0}^{n-1}g(F^{i}(x))=\int g \/d\mu
\end{equation*}
holds for all continuous maps  $g:\mathcal{D}\to \mathbb{R}$ is called the \emph{basin of $\mu$ for $F$}.  It
is well known that $B_{\mu}$ is  a Borel set (see for example \cite[Prop. 4.7]{DGS}). 
We call the measure $\mu$ \emph{physical for $F$} if $\mathcal{L}(B_\mu)>0$. 
An invariant measure $\hat{\nu}$ for the natural extension $\hat F:\hat{\mathcal{D}}\to \hat{ \mathcal{D}}$ is called an \emph{inverse limit physical measure} if $\hat{\nu}$ has a basin $\hat B_{\hat{\nu}}$ so that  $\mathcal{L}(\pi_0(\hat B_{\hat{\nu}}))>0$.

By results from \cite{KRS} if $\mathcal{L}$ is a physical measure for $F$  then the induced measure $\hat{\mathcal{L}}$ on $\hat{\mathcal{D}}$ is an inverse limit physical measure for the natural extension $\hat{F}$. In particular, there is a basin $\hat B_{\hat{\mathcal{L}}}:=\pi_0^{-1}(B)$ of $\hat{\mathcal{L}}$ for $\hat F$ with $\mathcal{L}(B)>0$.
In particular, this statement holds for $\lambda$, $I$ and $f$.

The pseudo-arc is a very curious object arising from Continuum Theory (see the survey of Lewis \cite{Lewis} and the introduction of \cite{BCO} for the overview of results involving the pseudo-arc).
Its complicated structure is reflected by the fact that it is {\em hereditarily indecomposable}, i.e., there are no proper subcontinua $A,B\subset H$ 
such that $A\cup B=H$ for every subcontinuum $H$ of the pseudo-arc $P$. In this sense its structure is in complete opposition to intuitive perception 
of an arc.
Let us recall two recent results relating the maps in $C_\lambda$ with pseudo-arc and planar attractors.

\begin{theorem*}\cite[Theorem 1.1]{CO}
The inverse limit with any $C_{\lambda}(I)$-generic map as the bonding map is the pseudo-arc.
\end{theorem*} 

\begin{theorem*}\cite[Theorem 1.6]{CO}
There exists a dense $G_{\delta}$ subset $G$ of $C_{\lambda}(I)$ and a parametrized family of homeomorphisms $\{\Phi_f\}_{f\in G}\subset \mathcal{H}(\mathcal{D}, \mathcal{D})$ varying continuously with $f$, having $\Phi_f$-invariant pseudo-arc attractors $\Lambda_f\subset D$ for every $f\in G$ so that 
 	\begin{enumerate}
		\item\label{item:a}  $\Phi_f|_{\Lambda_f}$ is topologically conjugate to $\hat{f}\colon \hat I_f\to \hat I_f$.
		\item\label{item:b} The attractors $\{\Lambda_f\}_{f\in G}$ vary continuously  in the Hausdorff topology. 
		\item\label{item:c} For each $f \in G$ the measure  $\mu_f := \hat{\lambda}$ induced by $(\lambda,f)$ is a weakly mixing, $\Phi_f$-invariant, inverse limit physical measure supported on the attractor $\Lambda_f$.\footnote{Here we use the fact that $\mathcal{D}$ is homeomorphic to some $\hat{\mathcal{D}}_{F_f}$
where maps $F_f$ vary continuously with $f$.} 
		\item The measures $\mu_f$ vary continuously in the weak* topology.
 	\end{enumerate}
\end{theorem*}

We can now interpret our results in this setting.  By Corollary~\ref{cor:ae} there is a set $X_f$ of full measure $\lambda$ such that $f|_{X_f}$ is invertible. Let
$Z_f:=\underleftarrow{\lim} (X_f,f)=\bigcap_{n=1}^\infty \pi_n^{-1}(X_f)$. Clearly
$\hat{\lambda}(Z_f)=1$ and $\pi_0|_{Z_f}$ is one-to-one onto $X_f$. Therefore, $\mu_f$ is isomorphic to $\lambda$ and from the ergodic theory point of view, dynamics of $(\Lambda_f,\Phi_f,\mu_f)$ and $(I,f,\lambda)$ are the same. Therefore, on one hand we observe very complicated topological structure of the attractor in the disc (pseudo-arc) and complicated topological dynamics on it (infinite topological entropy and topological mixing) but, on the other hand, from the physical perspective we see simple dynamics governed by $\lambda$ which is inherited from the interval
with zero  measure-theoretic entropy with respect to Lebesgue measure (yet weakly mixing).

\section{Acknowledgments}
We would like to thank Mark Pollicott for pointing out Parry's article \cite{Pa} to us.
The authors gratefully acknowledge the support of 
Slovenian research agency ARIS grant J1-4632. 
Our collaboration was also partially supported by the project AIMet4AI No.\ CZ.02.1.01/0.0/0.0/17\_049/0008414 and by the project from OP JAK, No. CZ.02.01.01/00/23\_021/0008759, which is co-financed by the European Union.

\begin{table}[ht]
	\begin{tabular}[t]{p{2.5cm}  p{11cm} }
		\includegraphics [width=2.1cm]{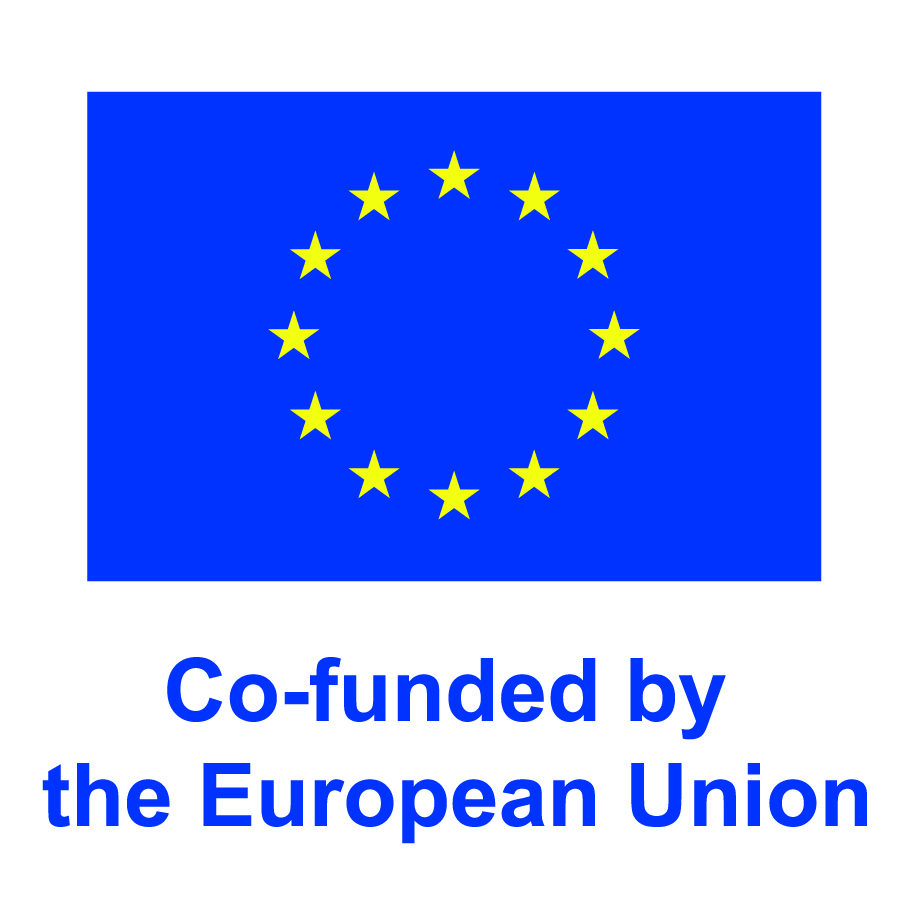} &
		\vspace{-2cm}
		This research is part of J.\ \v Cin\v c's  project that has received funding from the European Union's Horizon Europe research and innovation programme under the Marie Sk\l odowska-Curie grant agreement No.\ HE-MSCA-PF-PFSAIL-101063512.\\
	\end{tabular}
\end{table}


\begin{thebibliography}{99}

\bibitem{BCL} F.\ B\'eguin, S.\  Crovisier, F.\ Le Roux, 
{\em Construction of curious minimal uniquely ergodic homeomorphisms on manifolds: the Denjoy-Rees technique},
Annales Scientifiques de l'\'Ecole Normale Sup\'erieure,
{\bf  40}, 
(2007) 251--308.

\bibitem{B} J.\ Bobok, {\em On non-differentiable measure-preserving functions}, Real Analysis Exchange, {\bf 16} (1991) 119--129.

\bibitem{BT} J.\ Bobok, S.\ Troubetzkoy,  {\it Typical properties of interval maps
preserving the Lebesgue measure}, Nonlinearity \textbf{33} (2020), 6461--6501.

\bibitem{BCOT1} J.\ Bobok, J.\ \v Cin\v c, P.\ Oprocha, S.\ Troubetzkoy, {\em Periodic points and shadowing property of Lebesgue measure-preserving interval maps}, Nonlinearity {\bf 35} (2022) 2534--2557.

\bibitem{BCOT4} J.\ Bobok, J.\ \v Cin\v c, P.\ Oprocha, S.\ Troubetzkoy,  $s$-limit shadowing is generic for continuous Lebesgue measure-preserving circle maps. Ergodic Theory Dynam. Systems {\bf 43} (2023), no. 1, 78--98.

\bibitem{BCOT2} J.\ Bobok, J.\ \v Cin\v c, P.\ Oprocha, S.\ Troubetzkoy, {\em Interval maps with dense periodicity}, arXiv:2402.05638, February 2024.

\bibitem{BCOT3} J.\ Bobok, J.\ \v Cin\v c, P.\ Oprocha, S.\ Troubetzkoy, {\em Continuous Lebesgue measure-preserving maps on one-dimensional manifolds: a survey}, Topology Appl. 364 (2025), Paper No. 109101, 18 pp.

\bibitem{BCO} J.\ Boro\'nski, J.\ \v Cin\v c, P.\ Oprocha, {\em Beyond $0$ and $\infty$: The solution to the Barge entropy conjecture},
Tran. Amer. Math. Soc., posted on January 22, 2026, DOI: https://doi.org/10.1090/tran/9632 (to appear in print).

\bibitem{Bo} L.\ Bowen, {\em Zero entropy is generic},
Entropy {\bf 18} (2016) 20 pp.

\bibitem{BrGa} A.M.\ Bruckner, K.M.\ Garg, {\em The level structure of a residual set of continuous functions}, Trans. Amer. Math. Soc. {\bf 232}(1977), 307--321.

\bibitem{CC} S.L.\ Carvalho, A.\ Condori, {\em On the generic behavior of the metric entropy, and related quantities, of uniformly continuous maps over Polish metric spaces}, 
Mathematische Nachrichten
{\bf 296}(2023) 980--995.

\bibitem{CO} J. \v Cin\v c, P. Oprocha,
{\em Parametrized family of pseudo-arc attractors: physical measures and prime end rotations}, Proc.\ Lond.\ Math.\ Soc.\ {\bf (3)125} (2022) 318--357.

\bibitem{DGS} M. Denker, C. Grillenberger, K. Sigmund, Ergodic theory on compact spaces. Lecture Notes in Mathematics, Vol. 527. Springer-Verlag, Berlin-New York, 1976. {\rm iv}+360 pp.

\bibitem{FoWe} M.\ Foreman, B.\ Weiss, {\em An anti-classification theorem for ergodic measure preserving transformations}, J.\ Eur.\ Math.\ Soc.\  {\bf 6} (2004) 277--292.

\bibitem{FrWe} S.\ Friedland, B.\ Weiss, {\em Generalized interval exchanges and the 2--3 conjecture},
Central European J.\  Math.\ {\bf 3} 2005 412--429.

\bibitem{PAG} P.-A. Guih\'eneuf, 
{\em Propri\'et\'es dynamiques g\'en\'eriques des hom\'eomorphismes conservatifs.} 
Ensaios Mat., {\bf 22}
Sociedade Brasileira de Matemática, Rio de Janeiro, 2012. 115 pp.

\bibitem{KS1}  A.\ Katok, A.\ Stepin, {\em Approximations of ergodic dynamic systems by periodic
transformations}, Dolk.\ Akad.\ Nauk SSSR 171 (1966)  1268--1271.

\bibitem{KS2} A.\ Katok, A.\ Stepin, {\em Metric properties of measure preserving homeomorphisms}, Russ.\ Math.\ Surveys 25 (1970), 191--220.

\bibitem{KRS} J.\ Kennedy, B.\ E.\ Raines, D.\ R.\ Stockman, \emph{Basins of measures on inverse limit spaces for the induced homeomorphisms}, Ergod.\ Th.\ \& Dynam.\ Sys.\, {\bf 30} (2010), 1119--1130.

\bibitem{Lewis}
	W.\ Lewis, {\em The pseudo-arc.} Continuum theory and dynamical systems (Arcata, CA, 1989), 103–123, Contemp. Math., {\bf 117}, Amer.\ Math.\ Soc., Providence, RI, 1991.

\bibitem{Pa}  W.\ Parry, {\em  In general a degree two map is an automorphism}, Symbolic Dynamics and its 
Applications
(New Haven, CT, 1991) (Contemporary Mathematics, 135) Ed.\ P.\ Walters. American Mathematical
Society, Providence, RI, 1992, pp. 335--338.

\bibitem{Parth} K.\ R.\ Parthasarathy, Probability Measures on Metric Spaces. American Mathematical Society, Providence, RI, 2005. Reprint of the 1967 original.
\bibitem{PrUr} F.\ Przytycki, M. Urba\v nski, {\em Conformal Fractals - Ergodic Theory Methods}, London Mathematical
Society Lecture Note Series vol. 371, Cambridge University Press, 2010.

\bibitem{Ro} V.A.\ Rokhlin, {\em Entropy of metric automorphism}, Dokl.\ Akad.\ Nauk SSSR, 124 (1959) 980--983.

\bibitem{SW95} J.\ Schmeling, R.\ Winkler, \textit{Typical dimension of the graph of certain functions}, Monatsh. Math. {\bf 119} (1995), 303--320.

\bibitem{Si} K.\ Sigmund, {\em On the prevalence of zero entropy}, Israel J.\ Math.\ 10 (1971)  281--288.

\bibitem{Wal82} P.\ Walters, \emph{An introduction to ergodic theory}, Graduate Texts in
    Mathematics {\bf 79}, Springer Verlag, 1982, ix+250 pp.

\end{thebibliography}
\end{document}